\documentclass[12pt]{amsart}

\usepackage{ucs}

\usepackage{amssymb}
\usepackage{amsthm}
\usepackage{amsmath}
\usepackage{latexsym}
\usepackage[cp1251]{inputenc}
\usepackage{graphicx}
\usepackage{wrapfig}
\usepackage{caption}
\usepackage{subcaption}
\usepackage{indentfirst}
\usepackage[left=2.6cm,right=2.6cm,top=3cm,bottom=3cm,bindingoffset=0cm]{geometry}
\usepackage{enumerate}

\DeclareMathOperator{\aut}{Aut}

\DeclareMathOperator{\cay}{Cay}
\DeclareMathOperator{\cyc}{Cyc}

\DeclareMathOperator{\iso}{Iso}

\DeclareMathOperator{\orb}{Orb}

\DeclareMathOperator{\rk}{rk}

\DeclareMathOperator{\sym}{Sym}
\DeclareMathOperator{\rad}{rad}
\DeclareMathOperator{\poly}{poly}

\DeclareMathOperator{\WL}{WL}
\DeclareMathOperator{\Reg}{Reg}

\DeclareMathOperator{\qnrmaut}{QNRMAUT}
\DeclareMathOperator{\resolve}{RESOLVE}
\DeclareMathOperator{\pdbase}{PDBASE}
\DeclareMathOperator{\CRG}{CRG}
\DeclareMathOperator{\CGI}{CGI}
\DeclareMathOperator{\CGREC}{CGREC}

\makeatletter 
\def\@seccntformat#1{\csname the#1\endcsname. } 
\def\@biblabel#1{#1.}

\makeatother

\title{On Cayley representations of finite graphs over abelian $p$-groups}

\author{Grigory Ryabov}
\address{Novosibirsk State University, Novosibirsk, Russia}
\address{Sobolev Institute of Mathematics, Novosibirsk, Russia}
\email{gric2ryabov@gmail.com}
\thanks{The work is supported by the Russian Foundation for Basic Research (project 18-31-00051)}

\date{}

\newtheorem{prop}{Proposition}[section]
\newtheorem{theo}{Theorem}[section]

\newtheorem*{problem1}{Problem CGREC}

\newtheorem*{problem2}{Problem CGI}

\newtheorem*{problem3}{Problem CRG}

\newtheorem{lemm}[prop]{Lemma}

\theoremstyle{definition}

\begin{document}

\vspace{\baselineskip}
\vspace{\baselineskip}

\vspace{\baselineskip}

\vspace{\baselineskip}

\begin{abstract}
We construct a polynomial-time algorithm which given a graph $\Gamma$ finds the full set of non-equivalent Cayley representations of $\Gamma$ over the group $D\cong C_p\times C_{p^k}$, where $p\in\{2,3\}$ and $k\geq 1$. This result implies that the recognition and the isomorphism problems for Cayley graphs over $D$ can be solved in polynomial time.
\end{abstract}

\maketitle

\section{Introduction}

A \emph{Cayley representation} of a graph $\Gamma$ over a group $G$ is defined to be an isomorphism from $\Gamma$ to a Cayley graph over $G$ (here and further throughout the paper all the graphs and groups are assumed to be finite). Two Cayley representations of $\Gamma$ are called \emph{equivalent} if the images of $\Gamma$  under these representations are Cayley isomorphic, i.e., there exists a group automorphism of $G$ which is at the same time an isomorphism between the images. In the present paper we are interested in the following computational problem. 

\begin{problem3}
Given a group $G$ and a graph $\Gamma$  find a full set of non-equivalent Cayley representations of $\Gamma$ over $G$.
\end{problem3}

Here we assume that the group $G$ is given explicitly, i.e., by its multiplication table, and the graph $\Gamma$ is given by a binary relation. In the above form the Problem CRG was formulated in~\cite{NP}.

%The Problem CGR goes back to the Babai's work~\cite{Babai}. Namely, in~\cite[Problem~1.1]{Babai} Babai asked: given a group $G$ and a Cayley graph $\Gamma$ over $G$, are those of the form $\Gamma^{\alpha}$, $\alpha\in \aut(G)$, the only Cayley graphs over $G$ isomorphic to $\Gamma$? In our terms this means that $\Gamma$ has exactly one Cayley representation over $G$ up to equivalence. If every graph has at most one Cayley representation over a group $G$ up to equivalence then $G$ is called a \emph{CI-group}. CI-groups  have been extensively studied in many papers. We refer the reader to the survey~\cite{Li} for more information on CI-groups.

In general the Problem CRG seems to be very hard. Even the question whether a given graph has at least one Cayley representation over a given group leads to the recognition problem for Cayley graphs that can be formulated as follows. 

\begin{problem1}
Given a group $G$ and a graph $\Gamma$  test whether $\Gamma$ is isomorphic to a Cayley graph over $G$.
\end{problem1}

Another related problem is the isomorphism problem for Cayley graphs. In the following form this problem was formulated in~\cite{NP}.

\begin{problem2}
Given a group $G$, a Cayley graph over $G$, and an arbitrary graph test whether these two graphs are isomorphic.
\end{problem2}

For more information on the Problems CRG, CGREC, and CGI we refer the reader to~\cite{NP}.

One can check that the Problem CGI is reducible to the Problem CRG in polynomial time in the order of the group~$\aut(G)$. So if the group $G$ is generated by a set of at most constant size then the Problem CGI is polynomial-time reducible to the Problem CRG.

Denote the cyclic group of order~$n$ by $C_n$. The Problem CRG was solved efficiently for cyclic groups in~\cite{EP1} and for $C_2\times C_2\times C_p$, where $p$ is a prime, in~\cite{NP}. Up to now these results are the only published results concerned with solving the Problem CRG for infinite class of graphs. In the present paper we solve the Problem CRG for Cayley graphs over the group $D\cong C_p\times C_{p^k}$, where $p\in\{2,3\}$ and $k\geq 1$, in polynomial time. The above discussion implies that if the Problem CRG for $D$ can be solved in polynomial time then the Problems CGREC and CGI for $D$ also can be solved in polynomial time. The main result of the paper is given in the theorem below.

\begin{theo}\label{main}
For an explicitly given group $D\cong C_p\times C_{p^k}$ of order $n$, where $p\in\{2,3\}$ and $k\geq 1$, the Problems $\CRG$, $\CGREC$, and $\CGI$ can be solved in time $\poly(n)$.
\end{theo}

It should be mentioned that the Problem CGI in case when both graphs are Cayley graphs over a cyclic group was solved independently in~\cite{M2}. The Problem CGI in case when both graphs are Cayley graphs over $D$ was solved in~\cite{Ry2}.

Let $G$ be a finite group. The key notion used in the proof of Theorem~\ref{main} is a \emph{$G$-base} of a permutation group; by definition, this is a maximal set of pairwise non-conjugated regular subgroups isomorphic to $G$ of a permutation group. The notion of a $G$-base was suggested in~\cite{EMP} as a generalization of the notion of a \emph{cycle base} (see~\cite{EP1,M1}) which is, in fact, a $G$-base for a cyclic group $G$. One can check that all $G$-bases of a permutation group $K$ have the same size. Denote this size by $b_G(K)$.

Note that a graph $\Gamma$ is isomorphic to a Cayley graph over a group $G$ if and only if the group $\aut(\Gamma)$ contains a regular subgroup isomorphic to $G$. In other words, $\Gamma$ is isomorphic to a Cayley graph over $G$ if and only if $b_G(\aut(\Gamma))\neq 0$. The Babai argument (\cite{Babai}) implies that there is a one-to-one correspondence between regular subgroups of $\aut(\Gamma)$ isomorphic to $G$ and Cayley representations of $\Gamma$ over $G$. In addition, two Cayley representations are equivalent if and only if the corresponding subgroups are conjugate in $\aut(\Gamma)$. Therefore for solving the Problem CRG it is sufficient to find a $G$-base of $\aut(\Gamma)$. Thus, Theorem~\ref{main} is  an immediate consequence of the following theorem. 

\begin{theo}\label{main2}
Suppose that a group $D\cong C_p\times C_{p^k}$ of order $n$, where $p\in\{2,3\}$ and $k\geq 1$, is given explicitly. Then a $D$-base of the automorphism group of a graph on $n$ vertices can be found in time $\poly(n)$.
\end{theo}

Let us outline the proof of Theorem~\ref{main2}. Suppose that $\Gamma$ is a graph on $n$ vertices. Firstly we use the polynomial-time Weisfeiler-Leman algorithm~\cite{WeisL} to find the coherent configuration $\mathcal{X}$ (see Section~2 for exact definitions) corresponding to $\Gamma$ such that $\aut(\mathcal{X})=\aut(\Gamma)$. Put $K=\aut(\mathcal{X})$. A $D$-base of $K$ is not empty if and only if $\mathcal{X}$ is isomorphic to a Cayley scheme over~$D$. Further we use the classification of Cayley schemes over $D$  obtained in~\cite{MP} for $p=2$ and in~\cite{Ry1} for $p=3$ to construct efficiently a coherent configuration $\mathcal{X}^{'}$ such that (1) $K^{'}=\aut(\mathcal{X}^{'})$ is solvable; (2) $K^{'}\leq K$  and every $D$-base of $K^{'}$ contains a $D$-base of $K$ (Sections~3-5). The group $K^{'}$ is solvable and can be constructed efficiently. A $D$-base of $K^{'}$ is contained in a $D$-base of its Sylow $p$-subgroup $P$ and $P$ can be found by the polynomial-time Kantor's algorithm. In Section~6 we construct a polynomial-time algorithm for finding a $D$-base of a $p$-group. Applying this algorithm to  $P$, we obtain a $D$-base $B_D$ of $P$ containing a $D$-base of $K^{'}$ and hence a $D$-base of $K$. In Section~7 we summarize all above steps and show how to exclude from $B_D$ in polynomial time subgroups which are $K$-conjugate to other subgroup from $B_D$.

The author would like to thank prof. I Ponomarenko and prof. A. Vasil'ev for their valuable comments which allow to improve the text significantly.

~\

{\bf Notation.}

Given a finite set $\Omega$ the diagonal of $\Omega \times \Omega$ is denoted by $1_{\Omega}$.

For a set $T\subseteq 2^{\Omega}$ the set of all unions of the elements of $T$ is denoted by $T^{\cup}$.

If $s\subseteq \Omega \times \Omega$ and $S\subseteq 2^{\Omega \times \Omega}$ then set $s^{*}=\{(\beta,\alpha):(\alpha,\beta)\in s\}$ and $S^{*}=\{t^{*}:t\in S\}$.

Given $\alpha\in \Omega$ set $\alpha s=\{\beta\in \Omega: (\alpha,\beta)\in s\}$.

Given $r,s\subseteq \Omega \times \Omega$ set $rs=\{(\alpha,\gamma):(\alpha,\beta)\in r,~(\beta,\gamma)\in s~\text{for some}~\beta\in \Omega\}$.

Given $s\subseteq \Omega \times \Omega$ denote by $\langle s \rangle$ the equivalence closure of $s$, i.e. the smallest equivalence relation on $\Omega$ containing $s$.

If $E$ is an equivalence relation on $\Omega$ then the set of all classes of $E$ is denoted by $\Omega/E$.

%Given $\Delta\subseteq \Omega$ denote the set $\{\Lambda\in \Omega/E: \Lambda\cap \Delta \neq \varnothing\}$ by $\Delta/E$.

Given $s\subseteq \Omega \times \Omega$ set $s_{\Omega/E}=\{(\Lambda,\Delta)\in \Omega/E \times \Omega/E: s_{\Lambda,\Delta}\neq \varnothing\}$, where $s_{\Delta,\Lambda}=s\cap \Delta \times \Lambda$. Also set $s_{\Delta}=s_{\Delta,\Delta}$.

If $S\subseteq 2^{\Omega \times \Omega}$ and $\Delta\in \Omega_E$ then denote the sets $\{s_{\Omega/E}:s\in S, s_{\Omega/E}\neq \varnothing\}$ and $\{s_{\Delta}:s\in S, s_{\Delta}\neq \varnothing\}$ by $S_{\Omega/E}$ and $S_{\Delta}$ respectively.

The group of all permutations of a set $\Omega$ is denoted by $\sym(\Omega)$.

If $K\leq \sym(\Omega)$, $\alpha\in \Omega$, and $\Delta\subseteq \Omega$ then the one-point stabilizer of $\alpha$ and the setwise stabilizer of $\Delta$ in $K$ are denoted by $K_{\alpha}$ and $K_{\Delta}$ respectively.

The set of all orbits of $K\leq \sym(\Omega)$  is denoted by $\orb(K,\Omega)$.

If $K\leq \sym(\Omega)$ and $H$ is a group then the set of all regular subgroups of $K$ isomorphic to $H$ is denoted by $\Reg(K,H)$.

The set of non-identity elements of a group $G$ is denoted by  $G^\#$.

If $g\in G$ then the centralizer of $g$ in $G$ is denoted by $C_G(g)$.

If $H \leq G$ then  the centralizer and the normalizer of $H$ in $G$ are denoted by $C_G(H)$ and $N_G(H)$ respectively. 

The group $\{x\mapsto xg,~x\in G:g\in G\}$ of right translations  of $G$ is denoted by $G_{right}$.

Given $X\subseteq G$ denote by $s(X)$ the set $\{(g,xg): g\in G, x\in X\}\subseteq G\times G$ of edges of the Cayley graph $\cay(G,X)$.

For a set $\Delta\subseteq \sym(G)$ and a section $S=U/L$ of~$G$ set $\Delta^S=\{f^S:~f\in \Delta,~S^f=S\}$, where $S^f=S$ means that $f$ permutes the $L$-cosets in $U$ and $f^S$ denotes the bijection of $S$ induced by $f$.

The cyclic group of order $n$ is denoted by $C_n$.

\section{Coherent configurations}
In this section we give a background on coherent configurations. We use the notation and terminology from~\cite{CP}, where the most part of the material is contained. More information on coherent configurations can be found also in~\cite{EP2,NP}.

\subsection{Definitions}
Let $\Omega$ be a finite set of cardinality $n\geq 1$ and $S$ a partition of $\Omega\times \Omega$. A pair $\mathcal{X}=(\Omega,S)$ is called a \emph{coherent configuration} on $\Omega$ if $1_{\Omega}\in S^{\cup}$, $S^{*}=S$, and given $r,s,t\in S$ the number
$$c_{rs}^t=|\alpha r\cap \beta s^{*}|$$
does not depend on the choice of $(\alpha,\beta)\in t$. The elements of $\Omega$, elements of $S$, and numbers $c_{rs}^t$ are called the \emph{points}, \emph{basis relations}, and \emph{intersection numbers} of $\mathcal{X}$ respectively. The numbers $|\Omega|$ and $|S|$ are called the \emph{degree} and \emph{rank} of $\mathcal{X}$ respectively. Denote the rank of $\mathcal{X}$ by $\rk(\mathcal{X})$.   

The coherent configuration $\mathcal{X}=(\Omega,S)$ is said to be \emph{trivial} if $n=1$ or $\rk(\mathcal{X})=2$. We say that $\mathcal{X}$ is \emph{discrete} if $\rk(\mathcal{X})=n^2$, i.e. every element of $S$ is singleton. Denote the trivial and discrete coherent configurations on $\Omega$ by $\mathcal{T}_{\Omega}$ and $\mathcal{D}_{\Omega}$ respectively.

A set $\Delta \subseteq \Omega$ is called a \emph{fiber} of $\mathcal{X}$ if $1_{\Delta}\in S$. The set of all fibers of $\mathcal{X}$ is denoted by $F(\mathcal{X})$. Note that $\Omega$ is a disjoint union of all elements of $F(\mathcal{X})$. For every $r\in S$ there exist uniquely determined fibers $\Delta$ and $\Lambda$ such that $r\subseteq \Delta \times \Lambda$. This implies that $S$ is a disjoint union of the sets 
$$S_{\Delta,\Lambda}=\{r\in S:~r\subseteq \Delta \times \Lambda\}.$$
The number $c_{rr^{*}}^{1_{\Omega}}$ is called the \emph{valency} of $r$ and denoted by $n_r$. It is easy to see that $n_r=|\alpha r|$ for every $\alpha \in \Delta$. Given $T\in S^{\cup}$ the sum of all valences $n_t$, where $t$ runs over all basis relations inside $T$, is denoted by $n_T$. 

We say that $\mathcal{X}$ is  \emph{homogeneous} or $\mathcal{X}$ is a \emph{scheme} if $1_{\Omega}\in S$. If  $\mathcal{X}$ is a scheme then $n_r=n_{r^{*}}$ for every $r\in S$. We say that $\mathcal{X}$ is \emph{commutative} if $c_{rs}^t=c_{sr}^t$ for all $r,s,t\in S$. One can check that every commutative coherent configuration is a scheme.

The set of all equivalence relations $E\in S^{\cup}$ is denoted by $\mathcal{E}(\mathcal{X})$. The coherent configuration $\mathcal{X}$ is said to be \emph{primitive} if $\mathcal{E}(\mathcal{X})=\{1_{\Omega},\Omega\times \Omega\}$. It is easy to see that every primitive coherent configuration is a scheme. A scheme which is not primitive is said to be \emph{imprimitive}. One can verify that $\langle s \rangle \in \mathcal{E}(\mathcal{X})$ for every $s\in S^{\cup}$. 

Let $s\subseteq \Omega^2$. The largest relation $r\subseteq \Omega^2$ such that $sr=sr=s$ is called the \emph{radical} of $s$ and denoted by $\rad(s)$. Clearly, $1_{\Omega}\subseteq \rad(s)$ for every $s\subseteq \Omega^2$. One can check that if $s\in S^{\cup}$ then $\rad(s)\in \mathcal{E}(\mathcal{X})$.

\subsection{Isomorphisms}

Let $\mathcal{X}=(\Omega,S)$ and $\mathcal{X}^{'}=(\Omega^{'},S^{'})$ be coherent configurations. An \emph{algebraic isomorphism} from $\mathcal{X}$ to $\mathcal{X}^{'}$ is defined to be a bijection $\varphi:S\rightarrow S^{'}$ such that
$$c_{rs}^t=c_{r^{\varphi},s^{\varphi}}^{t^{\varphi}}$$
for every $r,s,t\in S$. In this case $\rk(\mathcal{X})=\rk(\mathcal{X}^{'})$, $|\Omega|=|\Omega^{'}|$, and $\mathcal{X}$ and $\mathcal{X}^{'}$ are homogeneous or not simultaneously. Every algebraic isomorphism is extended to a bijection from $\mathcal{E}(\mathcal{X})$ to $\mathcal{E}(\mathcal{X}^{'})$. This implies that $\mathcal{X}$ and $\mathcal{X}^{'}$ are primitive or not simultaneously.

An \emph{isomorphism} from $\mathcal{X}$ to $\mathcal{X}^{'}$ is defined to be a bijection $f:\Omega\rightarrow \Omega^{'}$ such that $S^{'}=S^f$, where $S^f=\{s^f:~s\in S\}$ and $s^f=\{(\alpha^f,~\beta^f):~(\alpha,~\beta)\in s\}$. In this case we say that $\mathcal{X}$ and $\mathcal{X}^{'}$ are \emph{isomorphic} and write $\mathcal{X}\cong\mathcal{X}^{'}$. The group $\iso(\mathcal{X})$ of all isomorphisms from $\mathcal{X}$ onto itself has a normal subgroup
$$\aut(\mathcal{X})=\{f\in \iso(\mathcal{X}): s^f=s~\text{for every}~s\in S\}.$$
This subgroup is called the \emph{automorphism group} of $\mathcal{X}$ and denoted by $\aut(\mathcal{X})$; the elements of $\aut(\mathcal{X})$ are called \emph{automorphisms} of $\mathcal{X}$. It is easy to see that if $\rk(\mathcal{X})=2$ then $\aut(\mathcal{X})=\sym(\Omega)$. If  $\aut(\mathcal{X})$ is transitive and $E\in \mathcal{E}(\mathcal{X})$ then the classes of $E$ are blocks of $\aut(\mathcal{X})$. Given $f\in\sym(\Omega)$ one can test whether $f\in \aut(\mathcal{X})$ in time $\poly(n)$ by testing for every $s\in S$ whether $s^f=s$.

Every isomorphism of coherent configurations induces in a natural way the algebraic isomorphism of them. However, not every algebraic isomorphism is induced by a combinatorial one (see~\cite[Section~4.2]{EP2}). A coherent configuration is called \emph{separable} if every algebraic isomorphism from it to another coherent configuration is induced by an isomorphism. Observe that $\mathcal{T}_{\Omega}$ and $\mathcal{D}_{\Omega}$ are separable.

\subsection{Restrictions and quotients}
Let $E\in \mathcal{E}(\mathcal{X})$ and $\Delta\in \Omega/E$. Then the pair
$$\mathcal{X}_{\Delta}=(\Delta, S_{\Delta})$$ 
is a coherent configuration called the \emph{restriction}  of $\mathcal{X}$ on $\Delta$. If $E$ is the union of $\Lambda \times \Lambda$, where $\Lambda \in F(\mathcal{X})$, and $\Delta\in F(\mathcal{X})$ then $\mathcal{X}_{\Delta}$ is called the \emph{homogeneous component} of $\mathcal{X}$. 
If $k\in \aut(\mathcal{X})_{\Delta}$ and $K\leq \aut(\mathcal{X})$ then denote by $k^{\Delta}$ and $K^{\Delta}$ the permutation induced by the action of~$k$ on $\Delta$ and the permutation group induced by the action of  $K_{\Delta}$ on $\Delta$ respectively. It is easy to see that
$$\aut(\mathcal{X})^{\Delta}\leq \aut(\mathcal{X}_{\Delta}).$$

Let $\mathcal{X}$ be  a scheme. Then the pair
$$\mathcal{X}_{\Omega/E}=(\Omega/E, S_{\Omega/E})$$ 
is a coherent configuration called the \emph{quotient}  of $\mathcal{X}$ modulo $E$. If $k\in \aut(\mathcal{X})$ and $K\leq \aut(\mathcal{X})$ then denote by $k^{\Omega/E}$ and $K^{\Omega/E}$ the permutation induced by the action of~$k$ on $\Omega/E$ and the permutation group induced by the action of $K$ on $\Omega/E$ respectively. Clearly, 
$$\aut(\mathcal{X})^{\Omega/E}\leq \aut(\mathcal{X}_{\Omega/E}).$$

Let $F\in \mathcal{E}(\mathcal{X})$ and $F\subseteq E$. Obviously, $E_{\Omega/F}\in \mathcal{E}(\mathcal{X}_{\Omega/F})$. It can be checked in a straightforward way that
$$\mathcal{X}_{(\Omega/F)/(E/F)}\cong \mathcal{X}_{\Omega/E}.~\eqno(1)$$ 
The relation $F_{\Delta}$ belongs to $\mathcal{E}(\mathcal{X}_{\Delta})$. The set $\Delta/F_{\Delta}$ is a class of the equivalence relation $E_{\Omega/F}$ which belongs to $\mathcal{E}(\mathcal{X}_{\Omega/F})$. One can check that $(\mathcal{X}_{\Delta})_{\Delta/F_{\Delta}}=(\mathcal{X}_{\Omega/F})_{\Delta/F_{\Delta}}$. The coherent configuration defined in this equality is denoted by $\mathcal{X}_{\Delta/F}$ and called a \emph{section} of $\mathcal{X}$. The sets of all sections of $\mathcal{X}$ and all primitive sections of $\mathcal{X}$ are denoted by $\mathcal{Q}(\mathcal{X})$ and $\mathcal{Q}(\mathcal{X})_{prim}$ respectively. 

If $k\in \aut(\mathcal{X})_{\Delta}$ and $K\leq \aut(\mathcal{X})$ then denote by $k^{\Delta/F}$ and $K^{\Delta/F}$ the permutation induced by the action of~$k$ on $\Delta/F_{\Delta}$ and the permutation group induced by the action of $K_{\Delta}$ on $\Delta/F_{\Delta}$ respectively. If $\mathcal{X}_{\Delta/F}$ is a section of $\mathcal{X}$ then 
$$\aut(\mathcal{X})^{\Delta/F}\leq \aut(\mathcal{X}_{\Delta/F}).$$

One can check  that for every $\Delta^{'}\in \Omega/E$ the bijection 
$$s_{\Delta}\mapsto s_{\Delta^{'}}$$
from $S_{\Delta}$ to $S_{\Delta^{'}}$ is an algebraic isomorphism from $\mathcal{X}_{\Delta}$ to $\mathcal{X}_{\Delta^{'}}$. So $\mathcal{X}_{\Delta^{'}/F}$ is algebraically isomorphic to  $\mathcal{X}_{\Delta/F}$ for every $\Delta^{'}\in \Omega/E$. In particular, $|\Delta^{'}/F_{\Delta^{'}}|=|\Delta/F_{\Delta}|$ and $\mathcal{X}_{\Delta^{'}/F}$ is primitive (of rank~2) if and only if $\mathcal{X}_{\Delta/F}$ is primitive (of rank~2). This implies the following statement.

\begin{lemm}\label{primitive}
Let $\mathcal{X}$ be a scheme and $E,F\in \mathcal{E}(\mathcal{X})$ such that $F\subset E$. Then there exists $R\in \mathcal{E}(\mathcal{X})$ with $F \subsetneq R  \subsetneq E$ if and only if there exists $\Delta \in \Omega/E$ such that $\mathcal{X}_{\Delta/F}$ is imprimitive.
\end{lemm}

Given a coherent configuration $\mathcal{X}$ on $\Omega$ denote by $\mathcal{F}(\mathcal{X})$ the set of all pairs  $(F,E)\in \mathcal{E}(\mathcal{X})^2$ such that $F\subseteq E$ and for every $\Delta \in \Omega/E$ the section $\mathcal{X}_{\Delta/F}$ has a composite degree and rank~2. Put 
$$m=\min \limits_{(F,E)\in \mathcal{F}(\mathcal{X})} |E|~\text{and}~\mathcal{F}_{\min}(\mathcal{X})=\{(F,E)\in \mathcal{F}(\mathcal{X}): |E|=m\}.$$  
Observe that $(F,E)\in \mathcal{F}(\mathcal{X})$ if and only if $F\subseteq E$ and for some $\Delta \in \Omega/E$ the section $\mathcal{X}_{\Delta/F}$ has a composite degree  and rank~2 because for every $\Delta^{'}\in \Omega/E$ the section $\mathcal{X}_{\Delta^{'}/F}$ is algebraically isomorphic to  $\mathcal{X}_{\Delta/F}$ and hence $\mathcal{X}_{\Delta^{'}/F}$ also has a composite degree and  rank~2.

\subsection{Wreath and tensor products}

Let $\mathcal{X}_1=(\Omega_1,S_1)$ and $\mathcal{X}_2=(\Omega_2,S_2)$ be coherent configurations. Put $S_1\otimes S_2=\{s_1\otimes s_2: s_1\in S_1,s_2\in S_2\}$, where $s_1\otimes s_2=\{((\alpha_1,\alpha_2),(\beta_1,\beta_2)): (\alpha_1,\beta_1)\in s_1, (\alpha_2,\beta_2)\in s_2\}$. Then the pair
$$\mathcal{X}_1\otimes \mathcal{X}_2=(\Omega_1\times \Omega_2, S_1\otimes S_2)$$
is a coherent configuration called the \emph{tensor product} of $\mathcal{X}_1$ and $\mathcal{X}_2$. It can be verified that 
$$\aut(\mathcal{X}_1\otimes \mathcal{X}_2)=\aut(\mathcal{X}_1)\times \aut(\mathcal{X}_2).$$

Let $\mathcal{X}=(\Omega,S)$ be a scheme and $E,F\in \mathcal{E}(\mathcal{X})$ with $F\subseteq E$. The scheme $\mathcal{X}$ is defined to be the \emph{$E/F$-wreath product} if $s\cap E=\varnothing$ implies that
$$s=\bigcup \limits_{(\Delta,\Lambda)\in s_{\Omega/F}} \Delta \times \Lambda$$ 
for every $s\in S$. Note that $F\subseteq \rad(s)$ for every $s\in S$ outside $E$. When the explicit indication of the equivalence relations $E$ and $F$ are not important we use the term \emph{generalized wreath product}. The $E/F$-wreath product is said to be \emph{trivial} if $F=1_{\Omega}$ or $E=\Omega\times \Omega$ and \emph{nontrivial} otherwise. Clearly, the nontrivial generalized wreath product is imprimitive.

Let $\Delta\in \Omega/E$. If $E=F$ and $\mathcal{X}_{\Delta}\cong \mathcal{X}_{\Delta^{'}}$ for every $\Delta^{'}\in \Omega/E$ then the generalized wreath product coincides with the standard wreath product of $\mathcal{X}_{\Delta}$ and $\mathcal{X}_{\Omega/E}$ (see~\cite[p.45]{Weis}).  In this case we write $\mathcal{X}=\mathcal{X}_{\Delta} \wr \mathcal{X}_{\Omega/E}$. One can check that if $\mathcal{X}=\mathcal{X}_{\Delta} \wr \mathcal{X}_{\Omega/E}$  then
$$\aut(\mathcal{X}_{\Delta}\wr \mathcal{X}_{\Omega/E})=\aut(\mathcal{X}_{\Delta})\wr \aut(\mathcal{X}_{\Omega/E}),$$
where the wreath product of two permutation groups in the right-hand side acts imprimitively.

\subsection{Algorithms}

A coherent configuration $\mathcal{X}=(\Omega,S)$ on $n$ points will always be given by the list of its basis relations. In this representation one can test in time $\poly(n)$ whether $\mathcal{X}$ is commutative, homogeneous, etc. Also in the same time one can list all elements of $F(\mathcal{X})$ and construct the restriction $\mathcal{X}_{\Delta}$ for every $\Delta\in F(\mathcal{X})^{\cup}$.

Let $s\subset \Omega \times \Omega$. The classes of $\langle s \rangle$ coincide with the connected components of the graph on $\Omega$ with the edge set $s\cup s^{*}$. So $\langle s \rangle$ can be constructed efficiently. Note that $\mathcal{X}$ is primitive if and only if $\langle s \rangle=\Omega\times \Omega$ for every nontrivial $s\in S$. Since $|S|\leq n^2$, one can test whether $\mathcal{X}$ is primitive in time $\poly(n)$.

If $E_1$ and $E_2$ are equivalences on $\Omega$ then $\langle E_1\cup E_2 \rangle$ is the smallest equivalence on $\Omega$ whose classes are unions of classes of $E_1$ and $E_2$. Since $\langle s \rangle\in \mathcal{E}(\mathcal{X})$ for every $s\in S^{\cup}$, every $E\in \mathcal{E}(\mathcal{X})\setminus \{1_{\Omega}\}$ is of the form $E=\langle E_1\cup s \rangle$, where $E_1$ is a maximal element of the set $\{E^{'}\in \mathcal{E}(\mathcal{X}):E^{'}\subset E, E^{'}\neq E\}$ and $s\in S$. Thus, all elements of $\mathcal{E}(\mathcal{X})$ can be listed in polynomial time in $n$ and $|\mathcal{E}(\mathcal{X})|$. 

Clearly, given $E\in \mathcal{E}(\mathcal{X})$ one can list all classes of $E$ and construct the quotient $\mathcal{X}_{\Omega/E}$ in time $\poly(n)$. Given $E,F\in \mathcal{E}(\mathcal{X})$ with $F\subseteq E$ and $\Delta\in \Omega/E$ the section $\mathcal{X}_{\Delta/F}$ also can be constructed in time $\poly(n)$.

We say that $\mathcal{X}$ is \emph{feasible} if $\mathcal{X}$ is  commutative and 
$$\mathcal{E}(\mathcal{X})=\{\langle r\cup s\rangle: r,s\in S\}.$$ 
Every feasible coherent configuration is a scheme because it is commutative. Observe that a commutative scheme $\mathcal{X}$ is feasible if and only if $\{\langle r \cup s\rangle: r,s\in S\}=\{\langle r\cup s \cup t\rangle: r,s,t\in S\}$.  The last condition can be verified in time $\poly(n)$ because $|S|\leq n^2$.  If $\mathcal{X}$ is feasible then the sets $\mathcal{E}(\mathcal{X})$, $\mathcal{Q}(\mathcal{X})$, and $\mathcal{Q}(\mathcal{X})_{prim}$ have the sizes  polynomial in~$n$. So the above discussion implies the following lemma.

\begin{lemm}\label{feastest}
Given a  coherent configuration $\mathcal{X}$ on $n$ points one can test in time $\poly(n)$ whether $\mathcal{X}$ is feasible and if so list all elements of $\mathcal{E}(\mathcal{X})$, $\mathcal{Q}(\mathcal{X})$, and $\mathcal{Q}(\mathcal{X})_{prim}$ within the same time.
\end{lemm}

 We finish this subsection with the lemma concerned with feasible schemes.

\begin{lemm}\label{maxpath}
Let $\mathcal{X}$ be a feasible scheme on $n$ points. Then one can find a maximal path $1_{\Omega}=E_0\subseteq \ldots \subseteq E_s=\Omega^2$ in $\mathcal{E}(\mathcal{X})$ in time $\poly(n)$.
\end{lemm}

\begin{proof}
Let $\Gamma$ be a directed graph with the vertex set $\mathcal{E}(\mathcal{X})$ and the edge set $\{(E_0,E_1)\in \mathcal{E}(\mathcal{X})^2: E_0 \subsetneq E_1\}$. Then $\Gamma$ is a directed acyclic graph. So one can find a maximal path in $\Gamma$ in polynomial time in $|\mathcal{E}(\mathcal{X})|$. Since $\mathcal{X}$ is feasible, we have $|\mathcal{E}(\mathcal{X})|\leq n^2$. Therefore a maximal path in $\Gamma$ can be found in time $\poly(n)$ and the lemma is proved.
\end{proof}

\subsection{Extensions}
The set of all coherent configurations on  $\Omega$ is partially ordered. Namely given coherent configurations $\mathcal{X}$ and $\mathcal{X}^{'}$ on $\Omega$ we set $\mathcal{X}\leq \mathcal{X}^{'}$ if and only if every basis relation of $\mathcal{X}$ is a union of some basis relations of $\mathcal{X}^{'}$. Clearly, the trivial and discrete coherent configurations are the minimal and maximal elements respectively. If  $\mathcal{X}\leq \mathcal{X}^{'}$ then $\aut(\mathcal{X})\geq \aut(\mathcal{X}^{'})$. If $E\in \mathcal{E}(\mathcal{X})$ and all $\mathcal{X}_{\Delta}$, $\Delta\in \Omega/E$, are pairwise isomorphic then the definition of the wreath product of coherent configurations yields that
$$\mathcal{X}\geq \mathcal{X}_{\Delta} \wr \mathcal{X}_{\Omega/E}.~\eqno(2)$$

Given a coherent  configuration $\mathcal{X}=(\Omega,S)$ and a set $T\subseteq 2^{\Omega\times \Omega}$   there exists the unique coherent configuration $\mathcal{Y}$ such that $\mathcal{Y}\geq \mathcal{X}$ and every element of $T$ is a union of some basis relations of  $\mathcal{Y}$. Moreover, $\mathcal{Y}$ can be constructed by the well-known Weisfeiler-Leman algorithm (see~\cite{Weis,WeisL}) in time polynomial in sizes of $T$ and $\Omega$. The coherent configuration $\mathcal{Y}$ is called the \emph{extension} of $\mathcal{X}$ with respect to $T$ and denoted by $\WL(\mathcal{X},T)$. 

\begin{lemm}\cite[Theorem~5.1]{NP}\label{autext}
Let $\mathcal{X}=(\Omega,S)$ be a coherent configuration, $T\subseteq 2^{\Omega\times \Omega}$, and $\mathcal{Y}=\WL(\mathcal{X},T)$. Then $\aut(\mathcal{Y})=\{f\in \aut(\mathcal{X}):t^f=t~\text{for every}~t\in T\}$.
\end{lemm}

\section{Cayley schemes}

\subsection{Definitions}

In this subsection we follow~\cite[Section~4.1]{NP} and~\cite[Section~2.4]{CP}. Let $G$ be a finite group and $e$ the identity of $G$. A coherent configuration $\mathcal{X}$ on the set $G$ is called a \emph{Cayley scheme} over $G$ if $\aut(\mathcal{X})\geq G_{right}$. In this case $\mathcal{X}$ is homogeneous because $G_{right}$ acts transitively on $G$. Clearly, if $G$ is abelian then $\mathcal{X}$ is commutative. If $G$ is cyclic then $\mathcal{X}$ is said to be \emph{circulant}. Every basis relation of $\mathcal{X}$ is an arc set of a Cayley graph over $G$. A coherent configuration is isomorphic to a Cayley scheme over $G$ if and only if its automorphism group contains a regular subgroup isomorphic to $G$. 

 One can check that $s(X)$ is an equivalence on $G$ if and only if $X$ is a subgroup of $G$. If $E\in \mathcal{E}(\mathcal{X})$ then the class of $E$ containing $e$ is denoted by $H_E$. It is easy to see that $H_E$ is a subgroup of $G$ and the classes of $E$ are the right $H_E$-cosets. Clearly,
$|\Omega/E|=|G/H_E|$ and $n_E=|H_E|$. Suppose that $F\in \mathcal{E}(\mathcal{X})$ and $F\subseteq E$. One can check that 
$$\mathcal{X}_{\Delta_1/F}\cong \mathcal{X}_{\Delta_2/F}~\eqno(3)$$ 
for every $\Delta_1,\Delta_2\in \Omega/E=G/H_E$. If $U=H_E$ and $L=H_F$  then put
$$\mathcal{X}_{U/L}=\mathcal{X}_{U/F}.$$
Observe that if $L$ is normal in~$U$ then $\mathcal{X}_{U/L}$ is a Cayley scheme over $U/L$ because $\aut(\mathcal{X}_{U/L})\geq \aut(\mathcal{X})^{U/L}\geq (G_{right})^{U/L}=(U/L)_{right}$. If $\mathcal{X}$ is the $E/F$-wreath product and $L$ is normal in $G$ then we say that $\mathcal{X}$ is also the $U/L$-wreath product. Put
$$\mathcal{H}(\mathcal{X})=\{H_E: E\in \mathcal{E}(\mathcal{X})\}.$$
A Cayley scheme $\mathcal{X}=(G,S)$ is said to be \emph{cyclotomic} if $S=\orb(KG_{right}, G^2)$ for some $K\leq \aut(G)$. In this case we write $\mathcal{X}=\cyc(K,G)$. If $\mathcal{X}$ is cyclotomic then $\mathcal{H}(\mathcal{X})$ contains all characteristic subgroups of $G$. One can check that a section of a cyclotomic Cayley scheme is also cyclotomic. We say that a Cayley scheme $\mathcal{X}$ is \emph{normal} if $G_{right}\trianglelefteq \aut(\mathcal{X})$.

\subsection{Cayley schemes over $C_{p^k}$ and $C_p\times C_{p^k}$}
  
Let $p$ be a prime and $k\geq 1$. Put $D=C\times B$, where $C=\langle c \rangle$, $|c|=p^k$, $B=\langle b \rangle$, $|b|=p$. If $l\leq k$ then put  $C_l=\{g\in C: |g|\leq p^l\}$ and $D_l=\{g\in D: |g|\leq p^l\}$. Throughout the paper  $\mathcal{K}_C$ and $\mathcal{K}_D$ denote the classes of schemes isomorphic to Cayley schemes over $C$ and $D$ respectively.

Let $\mathcal{X}$ be a Cayley scheme over $C$. Then $\mathcal{X}$ is feasible because every subgroup of $C$ is generated by one element. We say that a basis relation $s\in S$ is \emph{highest} if $\langle s \rangle=C^2$. It can be verified that all highest basic relations of $\mathcal{X}$ have the same radical (see~\cite{EP3}). The \emph{radical} $\rad(\mathcal{X})$ of $\mathcal{X}$ is defined to be the radical of a highest basis relation of $\mathcal{X}$.

Now let $\mathcal{X}$ be a Cayley scheme over $D$. In this case $\mathcal{X}$ is feasible because every subgroup of $D$ is generated by at most two elements. A basis relation $s\in S$ is said to be   \emph{highest} if  $\langle s \rangle=D^2$ or $|D/\langle s \rangle|=p$ and $\mathcal{X}_{H_{\langle s \rangle}}$ is circulant.  All highest basic relations of $\mathcal{X}$ have the same radical (see~\cite{Ry1}). The \emph{radical} $\rad(\mathcal{X})$  of $\mathcal{X}$ in this case also is defined to be the radical of a highest basis relation of $\mathcal{X}$.

Further we give a description of Cayley schemes over $C$ and $D$ in case when $p\in\{2,3\}$. If $\mathcal{X}$ is a scheme of degree~$p$, where $p\in\{2,3\}$, then $\rk(\mathcal{X})=2$ or $\mathcal{X}\cong \cyc(M,C_p)$, where $M$ is trivial (see~\cite{Han}). In both cases $\mathcal{X}\in \mathcal{K}_C$.

\begin{lemm}\label{circulant}
Let $p\in\{2,3\}$ and $\mathcal{X}$ a Cayley scheme over $C$. Then one of the following statements holds:

$(1)$ $\rk(\mathcal{X})=2;$

$(2)$ $\rad(\mathcal{X})=1_C$ and $\mathcal{X}$ is cyclotomic$;$

$(3)$ $\rad(\mathcal{X})>1_C$ and $\mathcal{X}$ is the nontrivial $U/L$-wreath product for some $U,L\in \mathcal{H}(\mathcal{X})$ such that $L\leq U$ and $\rad(\mathcal{X}_U)=1_U.$

\end{lemm}

\begin{proof}
Follows from~\cite[Theorem~4.1, Theorem~4.2]{EP3} and~\cite[Lemma~5.2]{Ry2}.
\end{proof}

The description of all Cayley schemes over $D$ was obtained, in fact, in~\cite{MP} for $p=2$ and in~\cite{Ry1} for $p=3$. The following lemma is taken from~\cite{Ry2}, where it was formulated in the language of $S$-rings.

\begin{lemm}\cite[Lemma~6.2]{Ry2}\label{element}
Let $p\in\{2,3\}$, $k=1$, and $\mathcal{X}$ a Cayley scheme over $D$. Then one of the following statements holds:

$(1)$ $\rk(\mathcal{X})=2;$

$(2)$ $\mathcal{X}$ is the tensor product of two  Cayley schemes over cyclic groups of order $p;$

$(3)$ $\mathcal{X}$ is the wreath product of two  Cayley schemes over cyclic groups of order $p;$

$(4)$ $p=3$ and $\mathcal{X}\cong\cyc(M,D)$, where $M=\langle \sigma \rangle$ and $\sigma:(c,b)\rightarrow (c^{-1},b^{-1});$

$(5)$  $p=3$ and $\mathcal{X}\cong \cyc(M,D)$, where $M=\langle \sigma \rangle$ and $\sigma:(c,b)\rightarrow(b,c^{-1}).$
\end{lemm}

If Statement~5 of Lemma~\ref{element} holds for a Cayley scheme $\mathcal{X}$ over $C_3\times C_3$ then $\mathcal{X}$ is called the \emph{Paley} scheme.

\begin{lemm}\label{sep}
Suppose that $\mathcal{X}$ is a scheme of degree~$p$, where $p\in\{2,3\}$, or $\mathcal{X}$ is the Paley scheme. Then the following hold:

$(1)$ $\mathcal{X}$ is  primitive, normal and separable;

$(2)$ $\aut(\mathcal{X})$ is solvable.
\end{lemm}

\begin{proof}
Let $\mathcal{X}$ be a scheme of degree~$p$, where $p\in\{2,3\}$. Then $\rk(\mathcal{X})=2$ or $\mathcal{X}\cong \cyc(M,C_p)$, where $M$ is trivial. In both cases $\mathcal{X}$ is  primitive. Also in both cases $(C_p)_{right}$ has index at most~2 in $\aut(\mathcal{X})$ and hence $\mathcal{X}$ is normal. In the former case $\mathcal{X}$ is obviously separable. In the latter case every basis relation of $\mathcal{X}$ has valency~1 and $\mathcal{X}$ is separable by~\cite[Theorem~3.3]{EP2}. Statement~2 of the lemma holds for $\mathcal{X}$ of degree~$p$ because $\aut(\mathcal{X})\leq \sym (p)$ and $\sym(p)$ is solvable for $p\in\{2,3\}$.

Suppose that $\mathcal{X}$ is the Paley scheme. Then $\mathcal{X}$ has degree~9 and rank~3. The straightforward check shows that $\mathcal{X}$ is primitive. Computer calculations made by using the GAP package COCO2P~\cite{GAP} show  that: (a) $\mathcal{X}$ is the unique up to an isomorphism primitive scheme of degree~9 and rank~3; (b) $\aut(\mathcal{X})=D_{right}\rtimes M^{'}$, where $D\cong C_3\times C_3$ and $M^{'}\cong C_4 \rtimes C_2$. Due to~(a) and~\cite[Theorem~1]{Ry2}, $\mathcal{X}$ is separable; due to~(b), $\mathcal{X}$ is normal and $\aut(\mathcal{X})$ is solvable. Thus, the lemma is proved.
\end{proof}

\begin{lemm}\label{Cayleyscheme}
Let $p\in\{2,3\}$, $k\geq 2$, and $\mathcal{X}$ a Cayley scheme over $D$. Then one of the following statements holds:

$(1)$ $\rk(\mathcal{X})=2;$

$(2)$ $\rad(\mathcal{X})=1_D$ and $\mathcal{X}=\mathcal{X}_V\otimes \mathcal{X}_S$ for some $V,S\in \mathcal{H}(\mathcal{X})$ such that $V\cong C_{p^k}$, $S\cong C_p$, $D=V\times S$, and $\rk(\mathcal{X}_V)=2;$

$(3)$  $\rad(\mathcal{X})=1_D$ and $\mathcal{X}$ is cyclotomic$;$

$(4)$ $\rad(\mathcal{X})>1_D$ and $\mathcal{X}$ is the nontrivial $U/L$-wreath product for some $U,L\in \mathcal{H}(\mathcal{X})$ such that $L\leq U$ and $\rad(\mathcal{X}_U)=1_U.$
\end{lemm}

\begin{proof}
Follows from~\cite[Lemma~6.3]{Ry2}. See also~\cite{MP,Ry1}.
\end{proof}

\begin{lemm}\label{primsection}
Let $p\in\{2,3\}$, $\mathcal{X}$ a Cayley scheme over $D$, and $\mathcal{X}^{'}\in \mathcal{Q}(\mathcal{X})_{prim}$. Then one of the following statements holds:

$(1)$ $\rk(\mathcal{X}^{'})=2;$

$(2)$ $\mathcal{X}^{'}$ has degree~$p;$

$(3)$ $\mathcal{X}^{'}$ is the Paley scheme.

\end{lemm}

\begin{proof}
In view of~(3), we may assume that $\mathcal{X}^{'}$ is a Cayley scheme over some section $U/L$ of $D$. If $|U/L|=p$ then Statement~2 of the lemma holds. Let $|U/L|\geq p^2$. Suppose that $U/L\cong C_{p^l}$ for some $l$. Then Lemma~\ref{circulant} holds for $\mathcal{X}^{'}$. If  $\mathcal{X}^{'}$ is cyclotomic then $\mathcal{H}(\mathcal{X}^{'})$ contains all characteristic subgroups of $U/L$, i.e. all subgroups of $U/L$, and hence $\mathcal{X}^{'}$ is imprimitive. If $\mathcal{X}^{'}$ is the generalized wreath product of two Cayley schemes then obviously $\mathcal{X}^{'}$ is imprimitive. Therefore $\rk(\mathcal{X}^{'})=2$ and Statement~1 of the lemma holds. 

Now suppose that $U/L\cong C_p\times C_{p^l}$ for some $l\geq 1$. If $|U/L|=p^2$ then Lemma~\ref{element} holds for $\mathcal{X}^{'}$. If one of the Statements~2-4 holds for $\mathcal{X}^{'}$ then obviously  $\mathcal{X}^{'}$ is imprimitive. Therefore $\rk(\mathcal{X}^{'})=2$ or $\mathcal{X}^{'}$ is the Paley scheme. So Statement~1 or Statement~3 of the lemma holds. 

If $|U/L|\geq p^3$ then Lemma~\ref{Cayleyscheme} holds for $\mathcal{X}^{'}$. If $\mathcal{X}^{'}$ is the tensor product or the generalized wreath product of two Cayley schemes then obviously $\mathcal{X}^{'}$ is imprimitive. If  $\mathcal{X}^{'}$ is cyclotomic then $\mathcal{H}(\mathcal{X}^{'})$  contains all characteristic subgroups of $U/L$, for example the proper subgroup of $U/L$ isomorphic to $C_p\times C_p$, and hence $\mathcal{X}^{'}$ is imprimitive. Therefore $\rk(\mathcal{X}^{'})=2$ and Statement~1 of the lemma holds. The lemma is proved.
\end{proof}

We finish this section with the following lemma which provides a special property of the automorphism group of a Cayley scheme over $D$ having a primitive  section of rank~2 and degree at least~$p^2$.  

\begin{lemm}\label{autcayley}
Let $p\in\{2,3\}$, $\mathcal{X}$ a Cayley scheme over $D$, $\mathcal{F}_{\min}(\mathcal{X})\neq \varnothing$, $(F,E)\in \mathcal{F}_{\min}(\mathcal{X})$, $U=H_E$, and $L=H_F$. Then $\aut(\mathcal{X})^{D/L}\geq \prod \limits_{\Delta\in D/U} \sym(\Delta/L)$.
\end{lemm}

Before we prove Lemma~\ref{autcayley}, we formulate and prove an auxiliary lemma.

\begin{lemm}\label{minsection}
In the conditions of Lemma~\ref{autcayley}, one of the following statements holds:

(1) $\rk(\mathcal{X}_U)=2;$

(2) $\mathcal{X}_U=\mathcal{X}_L\wr \mathcal{X}_{U/L}.$
\end{lemm}

\begin{proof}
If $|U|=p^2$ then $|L|=1$ because $(F,E)\in \mathcal{F}_{\min}(\mathcal{X})$. In this case Statement~1 of the lemma holds. Further we assume that $|U|\geq p^3$. The group $U$ is isomorphic to $C_{p^l}$ or $C_p\times C_{p^l}$ for some $l\leq k$. Firstly suppose that $\rad(\mathcal{X}_U)=1_U$. Then from Lemma~\ref{circulant} if $U\cong C_{p^l}$ and from Lemma~\ref{Cayleyscheme} if $U\cong C_p\times C_{p^l}$  it follows that 
$$\rk(\mathcal{X}_U)=2,~\text{or}~\mathcal{X}_U~\text{is cyclotomic, or}~\mathcal{X}_U=\mathcal{X}_V\otimes \mathcal{X}_S,$$ 
where $V<U$, $|V|\geq p^2$, and $\rk(\mathcal{X}_V)=2$. In the first case Statement~1 of the lemma holds. In the second case $\mathcal{H}(\mathcal{X}_{U/L})$ contains a nontrivial proper characteristic subgroup of $U/L$ because $|U/L|\geq p^2$. We obtain a contradiction because $\rk(\mathcal{X}_{U/L})=2$ and $\mathcal{H}(\mathcal{X}_{U/L})=\{\{L\}, U/L\}$. In the third case $|E_1|<|E|$, where $V=H_{E_1}$, $\rk(\mathcal{X}_V)=2$, and $|V|\geq p^2$. So $(F,E)\notin \mathcal{F}_{\min}(\mathcal{X})$, a contradiction.

Now suppose that $\rad(\mathcal{X}_U)>1_U$. Then $\mathcal{X}_U$ is the nontrivial generalized wreath product of two Cayley schemes by Lemma~\ref{circulant} if $U\cong C_{p^l}$ and by Lemma~\ref{Cayleyscheme} if $U\cong C_p\times C_{p^l}$.  Let $L_1=H_{F_1}$, where $F_1=\rad(\mathcal{X}_U)$. If $L_1\nleq L$ then $(F_1)_{U/F}\subseteq \rad(\mathcal{X}_{U/L})$ and $(F_1)_{U/F}$ is nontrivial, a contradiction with $\rk(\mathcal{X}_{U/L})=2$. So $L_1\leq L$. The scheme $\mathcal{X}_{U/L_1}$ has the trivial radical because otherwise $\rad(\mathcal{X}_U)$ is greater than $F_1$. Clearly, $\rk(\mathcal{X}_{(U/L_1)/(L/L_1)})=2$ and $|(U/L_1)/(L/L_1)|=|U/L|\geq p^2$. Lemma~\ref{circulant} in case $U/L_1\cong C_{p^m}$ and Lemma~\ref{Cayleyscheme} in case $U/L_1\cong C_p\times C_{p^m}$ implies that
$$\rk(\mathcal{X}_{U/L_1})=2,~\text{or}~\mathcal{X}_{U/L_1}~\text{is cyclotomic, or}~\mathcal{X}_{U/L_1}=\mathcal{X}_{V/L_1}\otimes \mathcal{X}_{S/L_1},$$ 
where $V/L_1<U/L_1$ and $\rk(\mathcal{X}_{V/L_1})=2$. Suppose that $\rk(\mathcal{X}_{U/L_1})=2$. Since $F_1$ is the radical of a highest basis relation of $\mathcal{X}_U$, there is exactly one basis relation of $\mathcal{X}_U$ outside $F_1$. So $L_1=L$ and Statement~2 of the lemma holds.  If  $\mathcal{X}_{U/L_1}$ is cyclotomic then $\mathcal{X}_{(U/L_1)/(L/L_1)}$ is also cyclotomic, a contradiction with $\rk(\mathcal{X}_{U/L_1})=2$. If $\mathcal{X}_{U/L_1}=\mathcal{X}_{V/L_1}\otimes \mathcal{X}_{S/L_1}$ then $(F_1,E_1)\in \mathcal{F}(\mathcal{X})$ and $|E_1|<|E|$ for $E_1=H_{V}$. This means that $(F,E)\notin \mathcal{F}_{\min}(\mathcal{X})$, a contradiction. The lemma is proved.
\end{proof}

\begin{proof}[Proof of the Lemma~\ref{autcayley}]
If $|D|=p^2$ then $U=D$, $L=e$, and $\rk(\mathcal{X})=2$. In this case $\aut(\mathcal{X})=\sym(D)$ and the lemma holds. Further we assume that $|D|\geq p^3$. Suppose that $\rad(\mathcal{X})=1_D$. Then one of the Statements~1-3 of Lemma~\ref{Cayleyscheme} holds for $\mathcal{X}$. If Statement~1 of Lemma~\ref{Cayleyscheme} holds for $\mathcal{X}$ then $\rk(\mathcal{X})=2$. So $L=e$, $\aut(\mathcal{X})=\sym(D)$, and hence the lemma holds. If Statement~2 of Lemma~\ref{Cayleyscheme} holds for $\mathcal{X}$ then $\mathcal{X}=\mathcal{X}_V\otimes \mathcal{X}_S$ for some $V,S\in \mathcal{H}(\mathcal{X})$ with $V\cong C_{p^k}$, $S\cong C_p$, $D=V\times S$, and $\rk(\mathcal{X}_V)=2$. Without loss of generality we may assume that $V=C$ and $S=B$. In this case $U=C$ and $L=e$ or $U=D$ and $L=B$. In the former case we obtain that
$$\aut(\mathcal{X})\geq\sym(C)\times B_{right}\geq \sym(C)\times \sym(Cb)\times \sym(Cb^{-1})$$
and the lemma holds. In the latter case $(F,E)\notin \mathcal{F}_{\min}(\mathcal{X})$ because $|U|<|D|$, $|U|\geq p^2$, and $\rk(\mathcal{X}_U)=2$. We obtain a contradiction with the assumption of the lemma.
%$$(\aut(\mathcal{X}))^{D/B}\geq (\sym(C)\times B_{right})^{D/B}=\sym(D/B)$$
 If Statement~3 of Lemma~\ref{Cayleyscheme} holds for $\mathcal{X}$ then $\mathcal{X}$ is cyclotomic and hence $\mathcal{X}_{U/L}$ is also cyclotomic, a contradiction with $\rk(\mathcal{X}_{U/L})=2$ and $|U/L|\geq p^2$. 

Now let $\rad(\mathcal{X})>1_D$. Then Lemma~\ref{Cayleyscheme} yields that $\mathcal{X}$ is the generalized wreath product of two Cayley schemes. Let $p^t=\max \limits_{g\in U} |g|$ and $D_t=\{g\in D: |g|\leq p^t\}\cong C_p\times C_{p^t}$. Clearly, $D_t=U$ or $U\cong C_{p^t}$ and $|D_t:U|=p$. Note that $D_t\in\mathcal{H}(\mathcal{X})$. Indeed, this is obvious if $U=D_t$ and follows from the description of Cayley schemes over $D$ given in Lemma~\ref{Cayleyscheme} otherwise. Let $E_1\in \mathcal{E}(\mathcal{X})$ such that $D_t=H_{E_1}$.

Let us prove that  
$$E\subseteq \rad(s)~\eqno(4)$$ 
for every basis relation $s$ of $\mathcal{X}$ outside $E_1$. Assume that there exists a basis relation $s$ outside $E_1$ with $E\nsubseteq \rad(s)$. From Lemma~\ref{minsection} it follows that there exists a basis relation $r$ of $\mathcal{X}$ such that $E=F\cup r$. Since $D_t\geq U$ and $s$ lies outside $E_1$, we conclude that $\langle s \rangle \cap r\neq \varnothing$. This yields that $r\subseteq \langle s \rangle$. Observe that $\langle r \rangle=E$. So $E\subseteq \langle s \rangle$. If $\rad(s)\cap r\neq \varnothing$ then $r\subseteq \rad(s)$ and hence $E=\langle r \rangle\subseteq \rad(s)$ which contradicts to our assumption. Therefore $\rad(s)\cap r=\varnothing$ and we have 
$\rad(s)\cap E=\rad(s)\cap F$. 

Let $U_1=H_{\langle s\rangle}$ and $L_1=H_{\rad(s)}$. The scheme $\mathcal{X}_{U_1/L_1}$ has the trivial radical. Since $\rad(s)\cap E=\rad(s)\cap F$, we obtain that $U\cap L_1=L\cap L_1$. This implies that $\pi(U)/\pi(L)\cong U/L$, where $\pi:D\rightarrow D/L_1$ is the canonical epimorphism. In particular, $|\pi(U)/\pi(L)|\geq p^2$. Also we have $\rk(\mathcal{X}_{\pi(U)/\pi(L)})=2$. Therefore $\mathcal{X}_{U_1/L_1}$ is a scheme with the trivial radical that has a section $\mathcal{X}_{\pi(U)/\pi(L)}$ of rank~2 and degree at least~$p^2$. Again, $\mathcal{X}_{U_1/L_1}$ can not be cyclotomic and hence $\rk(\mathcal{X}_{U_1/L_1})=2$ or $\mathcal{X}_{U_1/L_1}$ is the tensor product of a scheme of rank~2 and a scheme of degree~$p$. In both cases we have $\max \limits_{g\in\pi(U)} |g|=\max \limits_{g\in\pi(U_1)} |g|$. So $\max \limits_{g\in U} |g|=\max \limits_{g\in\ U_1} |g|$. This implies that $U_1\leq D_t$ and hence $s\subseteq E_1$. We obtain a contradiction with $s\nsubseteq E_1$. Thus,~(4) is proved.

Due to~(4) we conclude that $\mathcal{X}$ is the $D_t/U$-wreath product. If $D_t=U$ then $\mathcal{X}=\mathcal{X}_U\wr \mathcal{X}_{D/U}$.  If Statement~1 of Lemma~\ref{minsection} holds for $\mathcal{X}_U$ then $\rk(\mathcal{X}_U)=2$ and $L=e$. So
$$\aut(\mathcal{X})=\aut(\mathcal{X}_U)\wr \aut(\mathcal{X}_{D/U})\geq\sym(U)\wr (D/U)_{right}\geq \prod \limits_{\Delta\in D/U} \sym(\Delta)$$
and the lemma holds. If Statement~2 of Lemma~\ref{minsection} holds for $\mathcal{X}_U$ then $\mathcal{X}_U=\mathcal{X}_L\wr \mathcal{X}_{U/L}$. In this case we have
\begin{eqnarray}
\nonumber (\aut(\mathcal{X}))^{D/L}=((\aut(\mathcal{X}_L)\wr\aut(\mathcal{X}_{U/L}))\wr \aut(\mathcal{X}_{D/U}))^{D/L}\geq \\
\nonumber \geq ((\aut(\mathcal{X}_L)\wr \sym(U/L))\wr (D/U)_{right})^{D/L}\geq \prod \limits_{\Delta\in D/U} \sym(\Delta/L)
\end{eqnarray}
and the lemma also holds.

Consider the remaining case $|D_t:U|=p$. Put $K_0=\aut(\mathcal{X}_{D/U})$, $K_1=\aut(\mathcal{X}_{D_t})$, and for each $\Lambda, \Lambda^{'}\in D/D_t$ put $K_{\Lambda,\Lambda^{'}}=(D_tg^{-1})_{right}K_1(D_tg^{'})_{right}$, where $g,g^{'}\in D$ such that $D_tg=\Lambda$ and $D_tg^{'}=\Lambda^{'}$. Since $|D_t:U|=p\leq 3$, we have $K_0^{D_t/U}=K_1^{D_t/U}$. So $K_0$, $K_1$, and $K_{\Lambda,\Lambda^{'}}$ satisfy~(11) and~(12) from \cite[Section~5.2]{EP3}. Therefore $\aut(\mathcal{X})=K_1\wr_{D_t/U} K_0$ (see \cite[Definition~5.3, Theorem~5.4]{EP3}). 

Lemma~\ref{Cayleyscheme} and Lemma~\ref{minsection} imply that $\mathcal{X}_{D_t}=\mathcal{X}_U\otimes \mathcal{X}_S$ for some $S\in \mathcal{H}(\mathcal{X})$ with $|S|=p$ whenever $L=e$ and $\mathcal{X}_{D_t}=\mathcal{X}_L\wr(\mathcal{X}_{U/L}\otimes \mathcal{X}_{S/L})$ for some $S\in \mathcal{H}(\mathcal{X})$ with $|S/L|=p$ whenever $L>e$. This implies that $K_1=\sym(U)\times \aut(\mathcal{X}_S)$ or $K_1=\aut(\mathcal{X}_L)\wr (\sym(U/L) \times \aut(\mathcal{X}_{S/L}))$. In both cases 
$$K_1^{D_t/L}\geq \prod \limits_{\Delta\in D_t/U} \sym(\Delta/L).~\eqno(5)$$
Since $\aut(\mathcal{X})=K_1\wr_{D_t/U} K_0$, applying~(5) and~\cite[(7)]{EP3} to $\aut(\mathcal{X})$, we obtain that
$$\aut(\mathcal{X})^{D/L}\geq \prod \limits_{\Delta\in D/U} \sym(\Delta/L).$$
Thus, the lemma is proved.
\end{proof}

\section{Quasinormal schemes}

From now on until the end of the paper $\Omega$ is a set of size $n=p^{k+1}$, where $p$ is a prime  and $k\geq 1$. Let $p\in\{2,3\}$.  In view of Statement~1 of Lemma~\ref{sep}, each scheme of degree~$p$ and the Paley scheme are normal and primitive.  A feasible scheme $\mathcal{X}$ on the set $\Omega$ of size $p^{k+1}$, where $p\in\{2,3\}$ and $k\geq 1$, is said to be \emph{quasinormal} if for every $\mathcal{X}^{'}\in \mathcal{Q}(\mathcal{X})_{prim}$ one of the following statements holds: 

(1) $\mathcal{X}^{'}$ has degree $p$; 

(2) $\mathcal{X}^{'}$ is isomorphic to the Paley scheme.

\begin{lemm}\label{quasitest}
Given a coherent configuration $\mathcal{X}$ on $n=p^{k+1}$ points, where $p\in\{2,3\}$ and $k\geq 1$, one can test in time $\poly(n)$ whether $\mathcal{X}$ is a quasinormal scheme.
\end{lemm}

\begin{proof}
From Lemma~\ref{feastest} it follows that one can test whether $\mathcal{X}$ is feasible in time $\poly(n)$. If $\mathcal{X}$ is not feasible then it is not a quasinormal scheme. If $\mathcal{X}$ is feasible then the set $\mathcal{Q}(\mathcal{X})_{prim}$ of all primitive sections of $\mathcal{X}$ has the  size polynomial in~$n$. Lemma~\ref{feastest} implies that one can list in time $\poly(n)$ all elements of $\mathcal{Q}(\mathcal{X})_{prim}$. For every section from $\mathcal{Q}(\mathcal{X})_{prim}$ one can test in the constant time whether it has degree~$p$ or it is isomorphic to the Paley scheme. Thus, one can test whether $\mathcal{X}$ is a quasinormal scheme in time $\poly(n)$ and the lemma is proved.
\end{proof}

The main goal of this section is to show that for every feasible quasinormal scheme $\mathcal{X}$ of degree~$n$ the group $\aut(\mathcal{X})$ can be constructed in time $\poly(n)$. Firstly we show that there exists a solvable group $K$ containing $\aut(\mathcal{X})$ and $K$ can be constructed efficiently. Here and further throughout the paper a permutation group on $n$ points is always determined by a strong generating set containing at most~$n^2$ generators (see~\cite{Seress}).

\begin{center}
\textbf{Algorithm QNRMAUT}
\end{center}

\textbf{Input:} A quasinormal scheme $\mathcal{X}=(\Omega,S)$ of degree~$n=p^{k+1}$, where $p\in\{2,3\}$ and $k\geq 1$.

\textbf{Output:} A solvable group $K$ such that $K\geq \aut(\mathcal{X})$.

\textbf{Step 1.} Find a maximal path $1_{\Omega}=E_0\subseteq \ldots \subseteq E_m=\Omega^2$ in $\mathcal{E}(\mathcal{X})$ and for each $i\in \{0,\ldots,m\}$ choose $\Delta_i\in \Omega/E_i$ such that $\Delta_0\subseteq \ldots \subseteq\Delta_m=\Omega$.

\textbf{Step 2.} For each $i\in \{1,\ldots,m\}$ find the group $H_i=\aut(\mathcal{X}_{\Delta_i/E_{i-1}})$.

\textbf{Step 3.} Set $K_m=H_m$. For each $i=m-1,\ldots,1$ successively set  $K_i=H_i\wr K_{i+1}$.

\textbf{Step 4.} Output $K=K_1$.

\begin{prop}\label{qnrmaut}
Algorithm \textup{QNRMAUT} correctly constructs the group $K$ in time $\poly(n)$.
\end{prop}

\begin{proof}
The scheme $\mathcal{X}_{\Delta/E_{i-1}}$ is primitive for every $i\in\{1,\ldots,m\}$ and every $\Delta\in \Omega/E_i$. Indeed, if $\mathcal{X}_{\Delta/E_{i-1}}$ is not primitive for some $i$ and $\Delta\in \Omega/E_i$ then due to Lemma~\ref{primitive} there exists $E^{'}\in \mathcal{E}(\mathcal{X})$ such that $E_{i-1} \subsetneq E^{'}  \subsetneq E_i$. So $E_0\subseteq \ldots \subseteq E_m$ is not a maximal path, a contradiction. 

Since $\mathcal{X}$ is quasinormal, $\mathcal{X}_{\Delta/E_{i-1}}$ has degree $p$ or $\mathcal{X}_{\Delta/E_{i-1}}$ is isomorphic to the Paley scheme for every $i\in\{1,\ldots,m\}$ and every $\Delta\in \Omega/E_i$. For every $i\in\{1,\ldots,m\}$ and every $\Delta\in \Omega/E_i$ the coherent configuration $\mathcal{X}_{\Delta/E_{i-1}}$ is algebraically isomorphic to $\mathcal{X}_{\Delta_i/E_{i-1}}$. So $\mathcal{X}_{\Delta/E_{i-1}}\cong \mathcal{X}_{\Delta_i/E_{i-1}}$ for every $i\in\{1,\ldots,m\}$ and every $\Delta\in \Omega/E_i$ because each scheme of degree~$p$ and the Paley scheme are separable by Statement~1 of Lemma~\ref{sep}. This yields that on Step 3 each wreath product of permutation groups acting imprimitively is well-defined. Now applying (1) and (2) $m$ times we obtain that
$$\mathcal{X}\geq \mathcal{X}_{\Delta_1}\wr \mathcal{X}_{\Omega/E_1}\geq \ldots \geq \mathcal{X}_{\Delta_1} \wr (\mathcal{X}_{\Delta_2/E_1} \wr (\mathcal{X}_{\Delta_3/E_2} \wr \ldots (\mathcal{X}_{\Delta_{m-1}/E_{m-2}}\wr \mathcal{X}_{\Omega/E_{m-1}})\ldots )=\mathcal{Y}.$$ 
Clearly, $\aut(\mathcal{X})\leq \aut(\mathcal{Y})$. The definition of $K$ implies that $K=\aut(\mathcal{Y})$. So $K\geq \aut(\mathcal{X})$. The group $H_i$ is solvable for every $i\in\{1,\ldots,m\}$ by Statement~2 of Lemma~\ref{sep}. Therefore each $K_i$ is also solvable. In particular, $K=K_1$ is solvable.

From Lemma~\ref{maxpath} it follows that Step 1 requires time $\poly(n)$. For each $i\in \{1,\ldots,m-1\}$ the section $\mathcal{X}_{\Delta_i/E_{i-1}}$ can be constructed in polynomial time (see Subsection~2.5). Since $m\leq n^2$ and for each $i\in \{1,\ldots,m-1\}$ the section $\mathcal{X}_{\Delta_i/E_{i-1}}$ has degree at most~9, Step 2 can be done in time $\poly(n)$. Each $K_i$ is solvable and hence it can be constructed efficiently on Step~3. The proposition is proved.
\end{proof}

\begin{lemm}\label{qnrmaut2}
Let $\mathcal{X}$ be a quasinormal scheme of degree~$n=p^{k+1}$, where $p\in\{2,3\}$ and $k\geq 1$. Then the group $\aut(\mathcal{X})$ can  be found in time $\poly(n)$.
\end{lemm}

\begin{proof}
Let $K=\qnrmaut(\mathcal{X})$. Then $K\geq \aut(\mathcal{X})$, $K$ is solvable, and $K$ can be found in time $\poly(n)$ by Proposition~\ref{qnrmaut}. Now \cite[Theorem~8.4]{EP1} implies that the group $\aut(\mathcal{X})\cap K=\aut(\mathcal{X})$ also can be found in time $\poly(n)$.
\end{proof}

\section{Singular schemes}

 A feasible scheme $\mathcal{X}$ on the set $\Omega$ of size $p^{k+1}$, where $p\in\{2,3\}$ and $k\geq 1$, is said to be \emph{singular} if $\mathcal{F}(\mathcal{X})\neq \varnothing$. Clearly, $\mathcal{X}$ is singular if and only if $\mathcal{F}_{\min}(\mathcal{X})\neq \varnothing$. 

\begin{lemm}\label{singtest}
Given a coherent configuration $\mathcal{X}$ on $n=p^{k+1}$ points, where $p\in\{2,3\}$ and $k\geq 1$, one can test in time $\poly(n)$ whether $\mathcal{X}$ is a singular scheme and if so find within the same time the sets $\mathcal{F}(\mathcal{X})$ and $\mathcal{F}_{\min}(\mathcal{X})$.
\end{lemm}

\begin{proof}
Lemma~\ref{feastest} yields that  one can check whether $\mathcal{X}$ is feasible in time $\poly(n)$. If $\mathcal{X}$ is not feasible then it is not a singular scheme. If $\mathcal{X}$ is feasible then due to Lemma~\ref{feastest} one can find the set $\mathcal{E}(\mathcal{X})$ in time $\poly(n)$ and this set has the size  polynomial in~$n$. So one can test whether $\mathcal{F}(\mathcal{X})\neq \varnothing$ and if so  find the sets $\mathcal{F}(\mathcal{X})$ and $\mathcal{F}_{\min}(\mathcal{X})$  also in time $\poly(n)$. The lemma is proved. 
\end{proof}

Further we will show that for every singular scheme $\mathcal{X}$ one can construct in polynomial time a coherent configuration $\mathcal{Y}$ possessing the following properties: (1) $\mathcal{Y}>\mathcal{X}$; (2) the group $\aut(\mathcal{Y})$ controls regular subgroups from $\Reg(\aut(\mathcal{X}),D)$, i.e.  for every $G\in \Reg(\aut(\mathcal{X}),D)$ there exists $h\in \aut(\mathcal{X})$ such that $h^{-1}Gh\leq \aut(\mathcal{Y})$. Clearly, every $D$-base of $\aut(\mathcal{Y})$ contains a $D$-base of $\aut(\mathcal{X})$.

\begin{center}
\textbf{Algorithm RESOLVE}
\end{center}

\textbf{Input:} A singular scheme $\mathcal{X}=(\Omega,S)$ of degree~$n=p^{k+1}$, where $p\in\{2,3\}$ and $k\geq 1$, and $(F,E)\in \mathcal{F}_{\min}(\mathcal{X})$.

\textbf{Output:} A coherent configuration $\mathcal{Y}$ possessing properties (1)-(2).

\textbf{Step 1.} For every $\Delta\in \Omega/E$ choose a fixed-point-free permutation $c_{\Delta}\in \sym(\Delta/F_{\Delta})$ of order $p$.

\textbf{Step 2.} Put $R=\bigcup \limits_{\Delta\in \Omega/E} \bigcup \limits_{\Lambda\in \Delta/F_{\Delta}} \Lambda \times \Lambda^{c_{\Delta}}$.

\textbf{Step 3.} Output $\mathcal{Y}=\WL(\mathcal{X},\{R\})$.

\begin{prop}\label{resolve}
Algorithm \textup{RESOLVE} correctly constructs the coherent configuration $\mathcal{Y}$ in time $\poly(n)$.
\end{prop}

\begin{proof}
The definition of the extension implies that $\mathcal{Y}\geq \mathcal{X}$. Note that for every $\Delta \in \Omega/E$  the relation $\{(\Lambda, \Lambda^{c_{\Delta/F}}): \Lambda\in \Delta/F_{\Delta}\}\subseteq (\Delta/F_{\Delta})^2$ has valency~1 and it is a union of some basic relations of $\mathcal{Y}_{\Delta/F}$. So $\rk(\mathcal{Y}_{\Delta/F})>2$ and hence $\mathcal{Y}_{\Delta/F}>\mathcal{X}_{\Delta/F}$ for every $\Delta \in \Omega/E$. Therefore $\mathcal{Y}\neq \mathcal{X}$ and we conclude that  $\mathcal{Y}>\mathcal{X}$.

If $\Reg(\aut(\mathcal{X}),D)=\varnothing$ then  $\Reg(\aut(\mathcal{Y}),D)=\varnothing$ because $\aut(\mathcal{X})\geq \aut(\mathcal{Y})$. Now suppose that $\Reg(\aut(\mathcal{X}),D)\neq\varnothing$. This means that $\mathcal{X}\in \mathcal{K}_D$.  Let $G\in \Reg(\aut(\mathcal{X}),D)$. To prove the correctness of the algorithm it is sufficient to prove that $h^{-1}Gh\leq \aut(\mathcal{Y})$ for some $h\in \aut(\mathcal{X})$. Since $G$ is transitive and abelian, for every $\Delta\in \Omega/E$ the groups $G^{\Delta/F}$ and $G^{\Omega/E}$ are also transitive and abelian and hence they are regular. Denote by $G_0$ the kernel of the natural epimorphism from $G$ to $G^{\Omega/E}$. Observe that $G_0=G_{\Delta}$ for every $\Delta\in \Omega/E$ because $G^{\Omega/E}$ is regular. 

Let $\Delta_0 \in \Omega/E$. Choose $x_{\Delta_0}\in G^{\Delta_0/F}$ with $|x_{\Delta_0}|=p$. Let $x\in G_0=G_{\Delta_0}$ such that $x^{\Delta_0/F}=x_{\Delta_0}$. Since $G$ acts regularly on $\Omega/F$, we conclude that $x^{\Omega/F}$ is a product of disjoint cycles of the same length. This implies that $x^{\Omega/F}$ is a product of cycles of length $p$ because $|x^{\Delta_0/F}|=|x_{\Delta_0}|=p$. Therefore for every $\Delta\in \Omega/E$ the element $x^{\Delta/F}$ is a fixed-point-free permutation of order~$p$. So for every $\Delta\in \Omega/E$ there exists $h_{\Delta}\in \sym(\Delta/F_{\Delta})$ such that
$$h_{\Delta}^{-1}x^{\Delta/F}h_{\Delta}=c_{\Delta}.~\eqno(6)$$
Let $g\in G$. Then $g^{-1}xg=x$ because $G$ is abelian. So $g^{\Omega/F}$ permutes $x_{\Delta}$, $\Delta\in \Omega/E$, and
$$(g^{\Omega/F})^{-1}x^{\Delta/F}g^{\Omega/F}=x^{\Delta^g/F}~\eqno(7)$$
for every $\Delta\in \Omega/E$. 

Due to Lemma~\ref{autcayley}, there exists $h\in \aut(\mathcal{X})$ such that $h^{\Delta/F}=h_{\Delta}$. Put $G^{'}=h^{-1}Gh$. Let us prove that $G^{'}\leq \aut(\mathcal{Y})$. For every $\Delta\in \Omega/E$, every $\Lambda\in \Omega/F$ with $\Lambda\subseteq \Delta$, and every $g^{'}=h^{-1}gh\in G^{'}$ we have 
$$(\Lambda^{c_{\Delta}})^{g^{'}}=(\Lambda^{g^{'}})^{c_{\Delta^g}}.~\eqno(8)$$
Indeed,
\begin{eqnarray}
\nonumber (\Lambda^{c_{\Delta}})^{g^{'}}=(\Lambda^{c_{\Delta}})^{h^{-1}gh}=(\Lambda^{c_{\Delta}h^{-1}_{\Delta}})^{gh}=(\Lambda^{h^{-1}_{\Delta}x^{\Delta/F}})^{gh}=((\Lambda^{h^{-1}_{\Delta}})^g)^{g^{-1}x^{\Delta/F}gh}=\\
\nonumber =(\Lambda^{h^{-1}_{\Delta}})^{gx^{\Delta^g/F}h_{\Delta^g}}=(\Lambda^{h^{-1}_{\Delta}gh_{\Delta^g}})^{c_{\Delta^g}}=(\Lambda^{g^{'}})^{c_{\Delta^g}}.
\end{eqnarray}
In the above computation the third and the sixth equalities hold in view of~(6) and the fifth equality holds in view of~(7). Now using~(8), we obtain that
$$R^{g^{'}}=\bigcup \limits_{\Delta\in \Omega/E} \bigcup \limits_{\Lambda\in \Delta/F_{\Delta}} \Lambda^{g^{'}} \times (\Lambda^{c_{\Delta}})^{g^{'}}=\bigcup \limits_{\Delta\in \Omega/E} \bigcup \limits_{\Lambda\in \Delta/F_{\Delta}} \Lambda^{g^{'}} \times (\Lambda^{g^{'}})^{c_{\Delta^g}}=R$$
for every $g^{'}\in G^{'}$. Therefore $G^{'}\leq \aut(\mathcal{Y})$ by Lemma~\ref{autext}.

The Weisfeiler-Leman algorithm used on Step~3 requires time polynomial in~$n$ (see Subsection~2.6). So Algorithm RESOLVE requires time $\poly(n)$ and the proposition is proved.
\end{proof}

\section{Finding a $D$-base of a permutation group}

The main goal of this section is to show that a $D$-base of a permutation $p$-group can be found in polynomial time in the degree of this group. In this section $p$ is an arbitrary prime. Firstly we prove that a $D$-base of a permutation group of degree~$n$ has the  size polynomial in~$n$.

\begin{lemm}\label{basesize}
Let $K\leq \sym(\Omega)$. Then $b_D(K)\leq \frac{(p-1)^2 p!}{p^p} n^{p+2}$.
\end{lemm}

\begin{proof}
Consider the action of $K$ by conjugation on the set $W=\Reg(K,D)$. Let $W_1,\ldots,W_l$ be the orbits of this action and $D_i\in W_i$. Then $l=b_D(K)$. The stabilizer $K_{D_i}$ coincides with the normalizer $N_K(D_i)$. So $|W_i|=|K|/|N_K(D_i)|$. Since $N_K(D_i)/C_K(D_i)\leq \aut(D_i)$ and $D_i$ is generated by two elements of orders $p$ and $n/p$, we obtain that $|N_K(D_i)|/|C_K(D_i)|\leq |\aut(D_i)|\leq (p-1)^2n$. The centralizer of a regular group is regular. So $|C_K(D_i)|=n$ and $|N_K(D_i)|\leq (p-1)^2n^2$. This implies that $|W_i|\geq |K|/((p-1)^2n^2)$ and hence 

$$|W|=|W_1|+\ldots+|W_l|\geq l \frac{|K|}{(p-1)^2n^2}.~\eqno(9)$$ 

Now estimate the size of $W$. Every group $D^{'}$ from $W$ is generated by two elements $c^{'}$ and $b^{'}$ of degree $n$ such that $|c^{'}|=n/p$, $|b^{'}|=p$, and $b^{'}\in C_K(c^{'})$. The group $C_{\sym(\Omega)}(c^{'})$ is isomorphic to $C_{n/p}\wr \sym(p)$. This yields that $|C_K(c^{'})|\leq |C_{n/p}\wr \sym(p)|=(n/p)^p (p!)$. So for a fixed generator $c^{'}$ a generator $b^{'}$ can be chosen by at most $(n/p)^p (p!)$ ways. A generator $c^{'}$ can be chosen by at most $|K|$ ways. Therefore $|W|\leq |K|(n/p)^p (p!)$.  Thus from~(9) it follows that $l\frac{|K|}{(p-1)^2n^2} \leq |K|(n/p)^p (p!)$ and hence $l\leq \frac{(p-1)^2 p!}{p^p}n^{p+2}$. The lemma is proved.
\end{proof}

It should be mentioned that a cycle base of a permutation group of degree~$n$ has size at most $\varphi(n)$, where $\varphi$ is the Euler function (\cite{M1}).

The following can be found in~\cite{Seress}. Let $K\leq \sym(\Omega)$. Then one can check whether $K$ is transitive, primitive, regular in time $\poly(n)$. If $K$ is imprimitive then one can find the maximal and the minimal block systems for $K$ within the same time. Given a homomorphism $\psi: K\rightarrow \sym(\Omega^{'})$ and a set $M\subseteq K^{\psi}$ one can construct the groups $\ker(\psi)$, $K^{\psi}$, and the set $M^{\psi^{-1}}$ also in polynomial time. If $K$ is solvable and $K_1,K_2\leq K$ then one can find $C_K(K_1)$ and test whether $K_1$ and $K_2$ are $K$-conjugate in time $\poly(n)$. 

In \cite{EP1} it was proved that a cycle base of a solvable permutation group of degree~$n$ can be found in time $\poly(n)$. The next lemma directly follows from \cite[Theorem 6.1]{EP1}.

\begin{lemm}\label{conjbase}
Let $K\leq \sym(\Omega)$, $c\in\sym(\Omega)$, and $c^{-1}Kc=K$. Suppose that $\Omega=\Delta_0\cup \ldots \cup \Delta_{m-1}$ is a partition of $\Omega$, $(\Delta_i)^K=\Delta_i$, and $(\Delta_i)^c=\Delta_{i+1}$ for all $i$ modulo $m$. Then the set $X\subseteq Kc$ such that $Kc=\bigcup_{g\in K} X^g$ and $|X|\leq |K^{\Delta_0}|$ can be found (as the list of elements) in time $\poly(nr)$, where $r=|K^{\Delta_0}|$.
\end{lemm}

\begin{lemm}\label{sylow}
Let $K\leq \sym(\Omega)$ and $P$ a Sylow $p$-subgroup of $K$. The every $D$-base of $P$ contains a $D$-base of $K$.
\end{lemm}

\begin{proof}
If  $G\in \Reg(K,D)$ then the Sylow theorem implies that $h^{-1}Gh\leq P$ for some $h\in K$. So $b_D(P)\neq 0$ if and only if $b_D(K)\neq 0$. If $b_D(K)\neq 0$ then every $G\in \Reg(K,D)$ is $K$-conjugate to some group from a $D$-base of $P$ and we are done.
\end{proof}

 A Sylow $p$-subgroup of a permutation group of degree~$n$ can be found in time $\poly(n)$  by the Kantor algorithm (see~\cite{Seress}).  The algorithm below  constructs a $D$-base of a permutation $p$-group of degree~$n$ in time $\poly(n)$. In this algorithm we assume that $p$ is a constant.

\begin{center}
\textbf{Algorithm PDBASE}
\end{center}

\textbf{Input:} A permutation $p$-group $P\leq \sym(\Omega)$ of degree~$n=p^{k+1}$, where $p$ is a prime  and $k\geq 1$.

\textbf{Output:} A $D$-base $B_D$ of $P$, where $|D|=n$.

\textbf{Step 1.} If $P$ is not transitive then output $B_D=\varnothing$. 

\textbf{Step 2.} If $k=1$ then find $B_D$ by brute force and output $B_D$. If $k>1$ then find an imprimitivity system $\{\Delta_1,\ldots,\Delta_{n/p}\}$ of $P$ such that $|\Delta_i|=p$ for every $i\in \{1,\ldots,n/p\}$. Construct the groups $P^{\psi}$ and $\ker(\psi)$, where $\psi$ is the natural epimorphism from $P$ to $P$ acting on $\{\Delta_1,\ldots,\Delta_{n/p}\}$.

\textbf{Step 3.} Recursively find a $D_{k-1}$-base $B_{D_{k-1}}(P^{\psi})$ of $P^{\psi}$. Find  a $C$-base $B_{C}(P^{\psi})$ of $P^{\psi}$ and  the set $\overline{S}=\{\overline{h}\in B_{D_{k-1}}(P^{\psi}):|\overline{h}|=n/p^2\}\cup \{\overline{h}\in B_{C}(P^{\psi}):|\overline{h}|=n/p\}$. 

\textbf{Step 4.} For every $\overline{h}\in \overline{S}$ find $h\in P$ with $h^{\psi}=\overline{h}$ and then construct the set $X_h$ such that $|X_h|\leq p$ and $\ker(\psi)h=\bigcup\limits_{g\in \ker(\psi)} X_h^g$. Put $X=\bigcup \limits_{\overline{h}\in \overline{S}} X_h$. Find the set $T=\{x\in X:x~\text{is of degree}~n~\text{and}~|x|=n/p\}$.

 \textbf{Step 5.} For every $x\in T$ construct the set $Y_x=\{y\in C_P(x):|y|=p, y\notin \langle x \rangle\}$. Find the set $F=\{\langle x \rangle \times \langle y \rangle: x\in T, y\in Y_x\}$.

\textbf{Step 6.} For every $G\in F$ test whether $G$ is regular; if no then put $F=F\setminus \{G\}$. For every $G_1,G_2\in F$ test whether $G_1$ and $G_2$ are $P$-conjugate; if so put $F=F\setminus \{G_2\}$.

\textbf{Step 7.} Output $B_D=F$. 

\begin{prop}\label{dbase}
Algorithm \textup{PDBASE} correctly finds a $D$-base $B_D$ of $P$  in time $\poly(n)$.
\end{prop}

\begin{proof}
If $P$ is not transitive then $\Reg(P,D)=\varnothing$ and the algorithm terminates on Step~1. Suppose that $P$ is transitive. The imprimitivity system $\{\Delta_1,\ldots,\Delta_{n/p}\}$ on Step~2 exists because a $p$-group is primitive if and only if it is of order and degree~$p$. By the definition of $F$, after Step~6 we have $F\subseteq \Reg(P,D)$ and all groups from $F$ are pairwise nonconjugate in $P$. If $\Reg(P,D)=\varnothing$ then $F=\varnothing$. Now let $\Reg(P,D)\neq\varnothing$ and $G\in \Reg(P,D)$. To prove the correctness of the algorithm  it is sufficient to prove that $G$ is $P$-conjugate to some group from $F$. Let $g_1$ and $g_2$ be generators of $G$ of orders $n/p$ and $p$ respectively. The group $G^{\psi}$ is transitive and abelian and hence it is regular. Clearly, $G^{\psi}\cong D_{k-1}$ or $G^{\psi}\cong C$. So $G^{\psi}$ is $P^{\psi}$-conjugate to some group from $B_{D_{k-1}}(P^{\psi})\cup B_{C}(P^{\psi})$. This implies that $g_1^{\psi}$ is $P^{\psi}$-conjugate to some element from $\overline{S}$. Therefore $g_1$ is $P$-conjugate to some element $x\in T$. Let $h\in P$ such that $h^{-1}g_1h=x$. Since $g_2\in C_P(g_1)$, we obtain that $y=h^{-1}g_2h\in C_P(x)$ and hence $y\in Y_x$. Thus, $h^{-1}Gh\in F$.

Denote the running time of the algorithm applied to a group of degree~$n$ by $t(n)$. Let us prove that $t(n)$ is polynomial in~$n$. The discussion before Lemma~\ref{conjbase} yields that Steps~1-2  can be done in time $\poly(n)$. One can construct the set $B_{D_{k-1}}(P^{\psi})$ in time $t(n/p)$. From Lemma~\ref{basesize} it follows that $|B_{D_{k-1}}(P^{\psi})|\leq cn^{p+2}$ for $c=\frac{(p-1)^2 p!}{p^p}$. The set $B_{C}(P^{\psi})$ can be constructed in polynomial time by using Algorithm A3 from \cite{EP1} and $|B_{C}(P^{\psi})|\leq \varphi(n)$ by~\cite[Theorem~1.5]{M1}. Therefore Step~3 requires time $t(n/p)+\poly(n)$ and the set $S$ has the size polynomial in~$n$.  

Due to the discussion before Lemma~\ref{conjbase} for every $\overline{h}\in \overline{S}$ the element $h\in P$ with $h^{\psi}=\overline{h}$ can be found in time $\poly(n)$. By the definition of $\psi$, we have $(\Delta_i)^{\ker(\psi)}=\Delta_i$ for every $i\in \{1,\ldots,n/p\}$. Let $\overline{h}\in \overline{S}$ and $h\in P$ such that $h^{\psi}=\overline{h}$. Since  $\ker(\psi)$ is normal in $P$, we conclude that  $h^{-1}\ker(\psi)h=\ker(\psi)$. If $|\overline{h}|=n/p$ then without loss of generality we may assume that
$$\overline{h}=(\Delta_1 \ldots \Delta_{n/p}).$$ 
So $(\Delta_{i})^{h}=\Delta_{i+1}$ for every $i\in\{1,\ldots,n/p-1\}$ and $(\Delta_{n/p})^h=\Delta_1$. Therefore $\ker(\psi)$, $h$, and $\Delta_1,\ldots,\Delta_{n/p}$ satisfy the conditions of Lemma~\ref{conjbase}.
If $|\overline{h}|= n/p^2$ then without loss of generality we may assume that
$$\overline{h}=(\Delta_1 \ldots \Delta_{n/p^2})(\Delta_{n/p^2+1} \ldots \Delta_{2n/p^2})\ldots (\Delta_{(p-1)n/p^2+1} \ldots \Delta_{n/p}).$$
For every $i\in \{1,\ldots, n/p^2\}$ put 
$$\Lambda_i=\bigcup \limits_{j\in\{0,\ldots,p-1\}} \Delta_{i+jn/p^2}.$$ 
Then $\Lambda_i^h=\Lambda_{i+1}$ for every  $i\in\{1,\ldots,n/p^2-1\}$ and $\Lambda_{n/p^2}^h=\Lambda_1$. Therefore in this case $\ker(\psi)$, $h$, and $\Lambda_1,\ldots,\Lambda_{n/p}$ satisfy the conditions of Lemma~\ref{conjbase}. Now due to Lemma~\ref{conjbase} for every $\overline{h}\in \overline{S}$ the set $X_h$ can be constructed in time $\poly(n)$ and $|X_h|\leq |\ker(\psi)^{\Delta_1}|\leq p$. Since $\overline{S}$ has the polynomial size, the sets $X$ and $T$ have the polynomial sizes. Thus, Step~4 requires time $\poly(n)$. 

The group $P$ is solvable. So in view of the discussion before Lemma~\ref{conjbase}, Steps~5-6 require time $\poly(n)$. Thus, $t(n)\leq t(n/p)+\poly(n)$ and we are done by induction. The proposition is proved.
\end{proof}

\section{Main algorithm}

In this section we construct a polynomial-time algorithm for finding a $D$-base of the automorphism group of an arbitrary coherent  configuration in case when $|D|=n=p^{k+1}$, where $p\in\{2,3\}$ and $k\geq 1$.

\begin{lemm}\label{quassing}
Let $p\in\{2,3\}$. Then every scheme from $\mathcal{K}_D$ is quasinormal or singular.
\end{lemm}

\begin{proof}
Follows from Lemma~\ref{primsection}.
\end{proof}

\begin{center}
\textbf{Main algorithm}
\end{center}

\textbf{Input:} A coherent configuration $\mathcal{X}=(\Omega,S)$ of degree $n=p^{k+1}$, where $p\in\{2,3\}$, and $k\geq 1$.

\textbf{Output:} A $D$-base $B_D$ of $\aut(\mathcal{X})$, where $|D|=n$.

\textbf{Step 1.} Put $\mathcal{X}_0=\mathcal{X}$.

\textbf{Step 2.} If $\mathcal{X}$ is not feasible then output $B_D=\varnothing$.

\textbf{Step 3.} While $\mathcal{X}$ is singular do:

\begin{verse}
\textbf{Step 3.1.} Find a pair $(F,E)\in \mathcal{F}_{\min}(\mathcal{X})$;

\textbf{Step 3.2.} Put $\mathcal{X}=\resolve(\mathcal{X}, (F,E))$;

\textbf{Step 3.3.} If $\mathcal{X}$ is not feasible then output $B_D=\varnothing$.
\end{verse}

\textbf{Step 4.} If $\mathcal{X}$ is not quasinormal then output $B_D=\varnothing$.

\textbf{Step 5.} Find $K=\aut(\mathcal{X})$.

\textbf{Step 6.} If $p$ divides $|K|$ then find a Sylow $p$-subgroup $P$ of $K$; otherwise output $B_D=\varnothing$.

\textbf{Step 7.} Put $B=\pdbase(P)$.

\textbf{Step 8.} Output as $B_D$ a maximal subset of $B$ the elements of which are pairwise nonconjugate in $K_0=\aut(\mathcal{X}_0)$.

\begin{prop}\label{maintime}
The main algorithm correctly finds a $D$-base $B_D$ of $\aut(\mathcal{X})$  in time $\poly(n)$.
\end{prop}

\begin{proof}
Firstly suppose that $\mathcal{X}\notin \mathcal{K}_D$. Then $\Reg(\aut(\mathcal{X}),D)=\varnothing$. We may assume that the algorithm terminates on Step~8. The coherent configuration on Step~3 is greater than the input coherent configuration and hence $\Reg(K,D)=\Reg(P,D)=\varnothing$. Therefore the correctness of the algorithm follows from Proposition~\ref{dbase}. 

Now let $\mathcal{X}\in \mathcal{K}_D$. Then $\Reg(\aut(\mathcal{X}),D)\neq\varnothing$. So the algorithm does not terminate on Step~1. Since $\mathcal{X}\in \mathcal{K}_D$, the scheme $\mathcal{X}$ is feasible and hence the algorithm does not terminate on Step~2. Due to Proposition~\ref{resolve},  after each iteration on Step~3 we have 
$$\mathcal{X}>\mathcal{X}_0$$ 
and every $D$-base of $\aut(\mathcal{X})$ contains a $D$-base of $K_0$. This implies that $\mathcal{X}\in \mathcal{K}_D$ and hence  $\mathcal{X}$ is feasible after each iteration on Step~3.  So the algorithm does not terminate on Step~3. Clearly, $\mathcal{X}$ is not singular after Step~3. Therefore $\mathcal{X}$ is quasinormal on Step~4  by Lemma~\ref{quassing}. This yields that the algorithm does not terminate on Step~4. Since $\mathcal{X}\in \mathcal{K}_D$ on Step~6, $|K|$ is divisible by~$p$ and the algorithm does not terminate on Step~6. The set $B$ constructed on Step~7 is a $D$-base of $P$ by Proposition~\ref{dbase}. In view of Lemma~\ref{sylow} the set $B$ contains a $D$-base of $K=\aut(\mathcal{X})$. Every $D$-base of $K$ contains a $D$-base of $K_0$ (see Algorithm RESOLVE). Therefore $B$ contains a $D$-base of $K_0$. Thus, the set $B_D$ found on Step~8 is a $D$-base of $K_0$.

Now estimate the running time of the algorithm. Step~2 can be done in time $\poly(n)$ by Lemma~\ref{feastest}. Due to Lemma~\ref{singtest}, Step~3.1 requires polynomial time. Step~3.2 terminates in time $\poly(n)$ by Proposition~\ref{resolve}. Step~3.3 runs in polynomial time by Lemma~\ref{feastest}. Therefore each iteration on Step~3 requires polynomial time in~$n$. Since $\mathcal{X}$ is feasible on Step~3, we conclude that the set of all sections of $\mathcal{X}$ of rank~2 and degree at least~$p^2$ has the  size polynomial in~$n$. After each iteration on Step~3 the number of such sections becomes strictly less (Algorithm RESOLVE). So the number of iterations on Step~3 is polynomial in~$n$ and hence Step~3 terminates in time $\poly(n)$. From Lemma~\ref{quasitest} it follows that Step~4 can be done in time $\poly(n)$. Lemma~\ref{qnrmaut2} implies that Step~5 requires polynomial time. A Sylow $p$-subgroup $P$ of $K$ on Step~6 can be found by the polynomial-time Kantor algorithm (see~\cite{Seress}). In view of Proposition~\ref{dbase}, Step~7 can be done in time $\poly(n)$.

Now let us prove that the set $B_D$ on Step~8 can be found in time $\poly(n)$. The set $B$ has the size polynomial in $n$ by Lemma~\ref{basesize}. So to prove the required time bound for Step~8 it is sufficient to prove that given $G,G^{'}\in B$ one can check whether $G$ and $G^{'}$ are $K_0$-conjugate in time $\poly(n)$. Let $g=(\alpha_1\ldots \alpha_{n/p})\ldots (\alpha_{(p-1)n/p}\ldots \alpha_{n})\in G$ and $g^{'}=(\beta_1\ldots \beta_{n/p})\ldots (\beta_{(p-1)n/p}\ldots \beta_{n})\in G^{'}$ be elements of order $n/p$ and degree~$n$. Let $h_0$ be a permutation taking $\alpha_i$ to $\beta_i$ for each $i\in\{1,\ldots,n/p\}$. Then $h_0^{-1}gh_0=g^{'}$ and 
$$\{h\in \sym(\Omega): h^{-1}gh=g^{'}\}=C_{\sym(\Omega)}(g)h_0.$$
Note that $C_{\sym(\Omega)}(g)$ is permutationally isomorphic to the group $C_{n/p}\wr \sym(p)$ which is solvable. This yields that $C_{\sym(\Omega)}(g)$ can be constructed efficiently.  So the set $C_{\sym(\Omega)}(g)h_0$ has the size polynomial in~$n$ and it can be constructed in time $\poly(n)$. Therefore the set 
$$V(g,g^{'},K_0)=C_{\sym(\Omega)}(g)h_0\cap K_0=\{h\in K_0: h^{-1}gh=g^{'}\}$$ 
has the size polynomial in~$n$ and it can be constructed in time $\poly(n)$ by testing every permutation of $C_{\sym(\Omega)}(g)h_0$ for membership to the group $K_0$.

Let $g_1$ and $g_2$ be generators of $G$ of orders $n/p$ and $p$ respectively. Then the above discussion implies that one can construct the set 
$$V=\bigcup \limits_{g_1^{'}\in G^{'}, |g_1^{'}|=n/p} V(g_1,g_1^{'}, K_0)$$ 
in time $\poly(n)$. If $V=\varnothing$ then $G$ and $G^{'}$ are not $K_0$-conjugate. If $V\neq \varnothing$  then for every $h\in V$ one can check whether $h^{-1}g_2h\in G^{'}$. Since $V$ has the polynomial size, this can be done in polynomial time. If  $h^{-1}g_2h\in G^{'}$ for some $h\in V$ then $h^{-1}Gh=G^{'}$ and hence $G$ and $G^{'}$ are $K_0$-conjugate; otherwise $G$ and $G^{'}$ are not $K_0$-conjugate.
\end{proof}

\section{Proof Of Theorem~\ref{main2}}

\begin{proof}[Proof of the Theorem~\ref{main2}]
Given a graph $\Gamma$ on $n$ vertices one can construct by using the Weisfeiler-Leman algorithm (see \cite{Weis, WeisL}) in time $\poly(n)$ the coherent configuration $\mathcal{X}=\mathcal{X}(\Gamma)$ on $n$ points such that $\aut(\Gamma)=\aut(\mathcal{X})$. From Proposition~\ref{maintime} it follows that a $D$-base of $\aut(\mathcal{X})$ can be constructed in time $\poly(n)$. Therefore a $D$-base of $\aut(\Gamma)$  can be constructed in time $\poly(n)$ and the theorem is proved.
\end{proof}

\end{document}